\documentclass[10pt, bezier]{article}
\usepackage[toc,page]{appendix}
\usepackage{amsthm,amssymb}
\usepackage{amsmath}
\usepackage{graphicx,tikz,pgf}
\usepackage{color}
\usepackage{xcolor}
\usepackage{caption}
\usepackage{subcaption}
\usepackage{fancyhdr,cancel,enumerate,array,booktabs,setspace}
\usepackage[T1]{fontenc}
\usetikzlibrary{patterns,arrows,decorations.pathreplacing}
\usepackage{mathrsfs}
\usetikzlibrary{arrows}
\usepackage{epic}
\usepackage{eepic}
\usepackage{shadow}
\definecolor{xdxdff}{rgb}{0,0,0}
\definecolor{qqqqff}{rgb}{0,0,0}
\definecolor{uuuuuu}{rgb}{0,0,0}
\definecolor{ududff}{rgb}{0,0,0}
\date{ }
\theoremstyle{definition}

\newtheorem{definition}{Definition}
\newtheorem{remark}[definition]{Remark}
\newtheorem{example}[definition]{Example}

\newtheorem{theorem}[definition]{Theorem}

\newtheorem{lemma}[definition]{Lemma}

\newtheorem{corollary}[definition]{Corollary}

\title{\bf Some Mixed Graphs Determined by \linebreak Their Spectrum}
\author{\bf {S. Akbari$^{{\rm a}}$\footnote{E-mail addresses:
		s$\_$akbari@sharif.edu (S. Akbari), ghafaribaghestani$\_$a@mehr.sharif.edu (A. Ghafari), mnahvi2@illinois.edu (M. Nahvi), mohammadali.nematollahi69@student.sharif.edu (M.A. Nematollahi).}, A. Ghafari$^{{\rm a}}$, M. Nahvi$^{{\rm b}}$ and M.A. Nematollahi$^{{\rm a}}$}
\\[2mm]
${\rm ^{a}}$\small Department of Mathematical Sciences, Sharif University
of Technology, Tehran, Iran\\
${\rm ^{b}}$\small Department of Mathematics, University of Illinois at Urbana-Champaign, IL, USA}

\begin{document}

\maketitle
\begin{abstract}
	A mixed graph is obtained from a graph by orienting some of its edges. The Hermitian adjacency matrix of a mixed graph with the vertex set $ \{v_{1}, \ldots , v_{n}\} $, is the matrix $ H=[h_{ij}]_{n \times n} $, where $ h_{ij}=-h_{ji}=i $ if there is a directed edge from $ v_{i} $ to $ v_{j} $, $ h_{ij}=1 $ if there exists an undirected edge between $ v_{i} $ and $ v_{j} $, and $ h_{ij}=0 $ otherwise. The Hermitian spectrum of a mixed graph is defined to be the spectrum of its Hermitian adjacency matrix.
	
	In this paper we study mixed graphs which are determined by their Hermitian spectrum (DHS). First, we show that each mixed cycle is switching equivalent to either a mixed cycle with no directed edges ($C_{n}$), a mixed cycle with exactly one directed edge ($C_{n}^{1}$), or a mixed cycle with exactly two consecutive directed edges with the same direction ($C_{n}^{2}$) and we determine the spectrum of these three types of cycles. Next, we characterize all DHS mixed paths and mixed cycles. We show that all mixed paths of even order, except $P_{8}$ and $P_{14}$, are DHS. It is also shown that mixed paths of odd order, except $P_{3}$, are not DHS. Also, all cospectral mates of $P_{8}$, $P_{14}$ and $P_{4k+1}$  and two families of cospectral mates of $P_{4k+3}$, where $k\geq1$, are introduced. Finally, we show that the mixed cycles $C_{2k}$ and $C_{2k}^{2}$, where $k\geq3$, are not DHS, but the mixed cycles $C_{4}$, $C_{4}^{2}$, $C_{2k+1}$, $C_{2k+1}^{2}$, $C_{2k+1}^{1}$ and $C_{2j}^{1}$ except $C_{7}^{1}$, $C_{9}^{1}$, $C_{12}^{1}$ and $C_{15}^{1}$, are DHS, where $k\geq1$ and $j\geq2$.
\end{abstract}
\vskip 3mm

\vskip 3mm

\section{Introduction and Terminology}
In this paper all graphs we consider are simple and finite. A \textit{mixed graph} is obtained from an undirected graph by orienting a subset of its edges. Formally, a mixed graph $ X $ is given by its vertex set $ V(X) $, the set $ E_{0}(X) $ of undirected edges and the set $ E_{1}(X) $ of directed edges. So, $ E(X)= E_{0}(X) \cup E_{1}(X) $, where $E(X)$ is the edge set of $X$, and we distinguish undirected edges as unordered pairs $xy$ of vertices, while directed edges are shown as ordered pairs $(x,y)$ of vertices, where the direction of the edge is from $x$ to $y$. The underlying graph of a mixed graph $ X $ is denoted by $ G(X) $. \textit{Order}, \textit{size} and \textit{maximum degree} of a mixed graph are defined to be the order, size and maximum degree of its underlying graph, respectively. A mixed graph is called a \textit{mixed walk}, \textit{mixed path} or \textit{mixed cycle} if its underlying graph is a walk, path or
cycle, respectively. For $A \subseteq V(X)$, the induced graph on the vertices of $A$ is denoted by $\langle A\rangle$.

The \textit{Hermitian adjacency matrix} of a mixed graph $ X $ of order $n$, denoted by $H(X)_{n\times n}$, is given by
	\[
H(X)_{ij}=
\begin{cases}
1 & \text{if $v_{i}v_{j}\in E_{0}(X)$} \\
i & \text{if $(v_{i},v_{j}) \in E_{1}(X)$} \\
-i &  \text{if $(v_{j},v_{i}) \in E_{1}(X)$ }\\
0 & \text{otherwise}
\end{cases},
\]
 where $V(X)=\{v_{1},\ldots,v_{n}\}$ and $ i^{2}=-1 $. The Hermitian spectrum of $ X $, or simply the spectrum of $X$, denoted by $ Spec_{H}(X) $, is defined as the spectrum of $ H(X) $. It is evident that $ H(X) $ is a Hermitian matrix. Therefore, all its eigenvalues are real (See \cite[p.178]{lt}). The largest eigenvalue of $X$ is denoted by $\lambda_{1}(X)$. The mixed graph $ X $ is a \textit{cospectral mate} of the mixed graph $ Y $ if $ \text{Spec}_{H}(X)= \text{Spec}_{H}(Y) $. By \cite{guom}, two cospectral mates have the same order and size. The mixed graph $X$ is said to be $S$\textit{-out}, where $S\subset\mathbb{R}$, if $X$ has an eigenvalue not belonging to $S$.\\
 The \textit{value} of a mixed walk $W=v_{1}e_{12}v_{2} \cdots e_{(l-1)l} v_{l}$ is $ h(W)=h_{12}h_{23}\cdots h_{(l-1)l} $, where $H(W)=[h_{ij}]$. Note that for the walk $\overline{W}=v_{l}e_{l(l-1)}v_{l-1}\cdots v_{1}$, we have $h(\overline{W})=\overline{h(W)}$. Thus, if the value of a mixed cycle is $1$ (resp. $-1$) in one direction, then its
 value is the same for the reverse direction. We simply call such a mixed cycle a \textit{positive} (resp. \textit{negative}) mixed cycle. A mixed graph is called \textit{real} if all its cycles have real values. An \textit{elementary mixed graph} is a mixed graph whose each component is $K_{2}$ or a
 mixed cycle. The \textit{rank} and the \textit{corank} of a mixed
 graph $ X $, denoted by $r(X)$ and $s(X)$, respectively, are defined as
 \begin{equation*}
 r(X)=n-c \text{~~and~~} s(X)=m-n+c,
 \end{equation*}
 where $n$, $m$ and $c$ are the order, size and the number of components of $ X $, respectively.
 
 The characteristic polynomial of a mixed graph $ X $ is defined as $ \phi(X,\lambda)=\det (\lambda I-H(X))= \lambda^{n}+c_{1}\lambda^{n-1}+ \cdots + c_{n} $. We have
 \begin{theorem}\cite{liuli}\label{tcp}
 	Let $X$ be a mixed graph. The coefficients of $ \phi(X,\lambda)$ are given by
 	\begin{equation}\label{eccpm}
 	(-1)^{k}c_{k} =\sum_{H}(-1)^{r(H)+l(H)}2^{s(H)},
 	\end{equation}
 	where the summation is over all real elementary subgraphs $ H $ of $ X $ of order $ k $ and
 	$ l(H) $ denotes the number of negative mixed cycles of $ H $.
 \end{theorem}
 A \textit{signed graph} is a pair, say $\Gamma=(G,\sigma)$, where $ G $ is the underlying graph and $\sigma : E(G) \rightarrow \lbrace -1,+1 \rbrace$ is a sign function. Let $ \lbrace v_{1}, \ldots , v_{n} \rbrace $ be the vertices of $ \Gamma $. Then, the \emph{signed adjacency matrix} of $ \Gamma $, $ A(\Gamma)_{n\times n}$, or simply $ A $, is defined as $ A_{ij}= \sigma(v_{i}v_{j})a_{ij} $, where $ a_{ij}=1 $ if $ v_{i} $ and $ v_{j} $ are adjacent and $ a_{ij}=0 $, otherwise. The polynomial $ \phi(\Gamma,\lambda)=\det(\lambda I-A)=\lambda^{n}+a^{\sigma}_{1}\lambda^{n-1}+ \cdots + a^{\sigma}_{n-1}\lambda +a^{\sigma}_{n}  $ is called the \emph{characteristic polynomial} of $ \Gamma $. An \emph{elementary signed graph} is a signed graph whose each component is $K_{2}$ or a cycle. By arguments similar to unsigned graphs as stated in \cite{gre}, we have:
 \begin{equation}\label{eza}
 a^{\sigma}_{k}=\sum_{ U \in \, \mathcal{U}_{k}} (-1)^{\vert U \vert} 2^{t(U)} \sigma(U),
 \end{equation}
 where $ \mathcal{U}_{k} $ is the set of all elementary subgraphs of $ \Gamma $ of order $ k $, $ \vert U \vert $ is the number of components of $ U $, $ t(U) $ is the number of cycles in $U$ and, $ \sigma(U) $ is the product of the sign of all cycles of $U$ (if $ U $ has no cycles, then we define $ \sigma(U)=1 $). A connected graph with a unique cycle is called a unicyclic graph. The signed unicyclic graph with exactly one negative edge in its cycle is denoted by $(G,-)$, where $G$ is its underlying graph.
 
 A \textit{switching function} on a mixed graph $ X $ is a function $ \theta : V(X) \rightarrow \{ \pm 1, \pm i\} $. \textit{Switching} a mixed graph $ X $ to a mixed graph $ Y $ means that there exists a diagonal matrix $ D(\theta) = \text{diag}(\theta(v_{1}),\ldots,\theta(v_{n})) $ such that $ H(X)= D(\theta)^{-1}H(Y)D(\theta) $, where $V(X)=\{v_{1},\ldots,v_{n}\}$, and we say that $ X $ and $ Y $ are \textit{switching equivalent}. In other words, by switching a mixed graph, in each step, for a certain vertex, say $u$, we multiply the value of all edges in the form $uv$ and $(u,v)$, where $v$ is adjacent to $u$, by one of $1$, $-1$, $i$ or $-i$. It is straightforward to see that if two mixed graphs $ X $ and $ Y $ are switching equivalent, then $ \text{Spec}_{H}(X)= \text{Spec}_{H}(Y) $. A mixed graph $ X $ is \textit{determined by its Hermitian
spectrum}, or DHS, if all its cospectral mates can be obtained
from $ X $ by switching.

The path, the cycle and the complete graph of order $n$ are denoted by $P_{n}$, $C_{n}$ and $K_{n}$, respectively. A $\theta$-\textit{graph}, denoted by $\theta_{p,q,r}$, consists of three internally vertex-disjoint paths $P_{p}$, $P_{q}$ and $P_{r}$ with common endpoints, where $p,q,r\geq2$. The graph $\theta_{p,q,r}$ is \textit{proper} if $p,q,r\geq3$.

As simple graphs, paths and cycles are determined by their spectrum (DS) \cite{haemers-p}. The spectral determination problem of signed paths and signed cycles has been studied in \cite{ss} and \cite{sop}. In this paper, we study the same problem for mixed paths and mixed cycles. We start by showing that there are three types of mixed cycles up to switching equivalence. Next, we show that $P_{3}$ and paths of even order except $P_{8}$ and $P_{14}$ are the only DHS mixed paths. We also characterize all cospectral mates of $P_{8}$, $P_{14}$ and $P_{4k+1}$, where $k\geq1$, and two families of cospectral mates of $P_{4k+3}$, where $k\geq1$, are introduced. Finally, all DHS mixed cycles are characterized. Note that as mixed graphs, the paths of order $ n $, where $ n \equiv 1 \pmod 4 $, are not DHS, but as signed graphs, except $ P_{13} $, $ P_{17} $ and $ P_{29} $, they are determined by their spectrum \cite{sop}. The final section of this paper is the Appendix A, which includes two figures. The mixed graphs in the first figure are $(-2,2)$-out and the mixed graphs in the second figure are possible components of a mixed cospectral mate of a path. Also, the Appendix A includes the spectrum of the graphs given in the second figure.

We conclude this section by stating a theorem which will be a useful tool in proving our results.
\begin{theorem}\label{forest}
	\cite{liuli}
	Let $ F $ be a forest. Then all mixed graphs whose underlying graph is isomorphic to $ F $ are switching equivalent to $ F $.
\end{theorem}
\section{Hermitian Spectrum of Mixed Cycles}\label{s2}
In this section, we study the spectral theory of mixed cycles. First, we examine the Hermitian spectrum of mixed cycles using the switching function discussed in Section 1. Here, our approach is completely algorithmic and intuitive.

Let $ X $ be a mixed cycle. We will show that up to switching equivalence, there are exactly three types of mixed cycles. To see this we consider four switching functions:
\begin{enumerate}
	\item[Sw.1.] Let $ \{ u , v , w\} \subseteq V(X) $ and $ (u,v) , (v,w) \in E_{1}(X) $. Define the switching function
	\begin{equation*}
	\theta(v)=-1 \text{~~and~~} \theta(x) = 1,\qquad\text{~~for every~~} x \in V(X)\backslash\{v\}.
	\end{equation*}
	Then, $ X $ changes to a mixed cycle such that $ (v,u) , (w,v) \in E_{1}(X) $ and all other edges remain unchanged. So, this action reverses the orientation of two consecutive directed edges with the same direction. 
	\item[Sw.2.] Let $ \{ u , v , w\} \subseteq V(X) $ and $ (u,v) , (w,v) \in E_{1}(X) $. Define the switching function
	\begin{equation*}
	\theta(v)=i \text{~~and~~} \theta(x) = 1,\qquad\text{~~for every~~} x \in V(X)\backslash\{v\}.
	\end{equation*}
	Then, $ X $ changes to a mixed cycle such that $ uv , vw \in E_{0}(X) $ and all other edges remain unchanged.
		\item[Sw.3.] Let $ \{ u , v , w\} \subseteq V(X) $ and $ (v,u) , (v,w) \in E_{1}(X) $. Define the switching function
	\begin{equation*}
	\theta(v)=-i \text{~~and~~} \theta(x) = 1,\qquad\text{~~for every~~} x \in V(X)\backslash\{v\}.
	\end{equation*}
	Then, $ X $ changes to a mixed cycle such that $ uv , vw \in E_{0}(X) $ and all other edges remain unchanged.
		\item[Sw.4.] Let $ \{ u , v , w\} \subseteq V(X) $ and $ (u,v) \in E_{1}(X)$ and $ vw \in E_{0}(X) $. Define the switching function
	\begin{equation*}
	\theta(v)=i \text{~~and~~} \theta(x) = 1,\qquad\text{~~for every~~} x \in V(X)\backslash\{v\}.
	\end{equation*}
	Then, $ X $ changes to a mixed cycle such that $ uv \in E_{0}(X)$ and $  (v,w) \in E_{1}(X) $ and all other edges remain unchanged.
\end{enumerate}
Now, after examining all possible cases, we are ready to establish the following theorem.
\begin{theorem}
	\label{Types}
	Let $ X $ be a mixed cycle of order $ n $. Then, $ X $ is switching equivalent to exactly one of the three following mixed cycles:
	\begin{enumerate}
		\item[1.] $ C_{n} $: The mixed cycle whose all edges are undirected.
		\item[2.]  $ C_{n}^{1} $: The mixed cycle which has exactly one directed edge.
		\item[3.]  $ C_{n}^{2} $: The mixed cycle which has two consecutive directed edges with the same direction and all its other edges are undirected.
	\end{enumerate}
	We call them mixed cycles of \textit{Type} $0$, $1$ and $2$, respectively.
\end{theorem}
\begin{proof}
	Let $X$ be a mixed cycle. By the four switching functions introduced above, there exists a mixed cycle switching equivalent to $X$, in which all directed edges are consecutive and have the same direction. Let $(u_{1},u_{2}),(u_{2},u_{3}),\ldots,(u_{l},u_{l+1})$ be all directed edges of such a mixed cycle. If $l\leq2$, then the proof is complete. So, let $l\geq3$. Using Sw.1, we can reverse the direction of $(u_{l-2},u_{l-1})$ and $(u_{l-1},u_{l})$. Then, using Sw.2 and Sw.3, we reduce the number of directed edges. By repeating this procedure, one can see that $X$ is switching equivalent to one of the mixed cycles mentioned above.
\end{proof}
\begin{example}
	In Figure \ref{switch}, a mixed cycle is switched to a mixed cycle of Type 1.
\end{example}
\begin{figure}[!h]
	\centering
	\includegraphics[height=2.25cm]{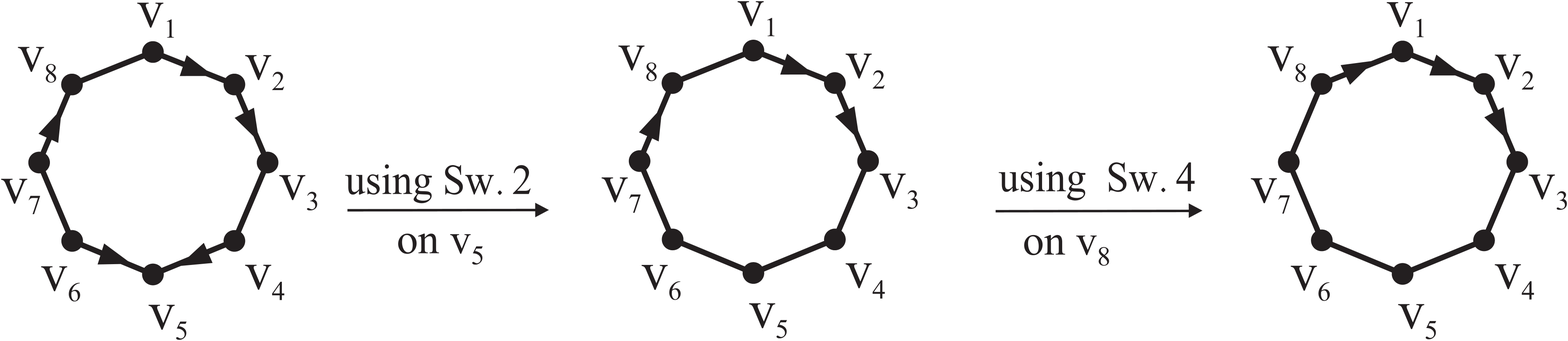}
	\includegraphics[height=2.25cm]{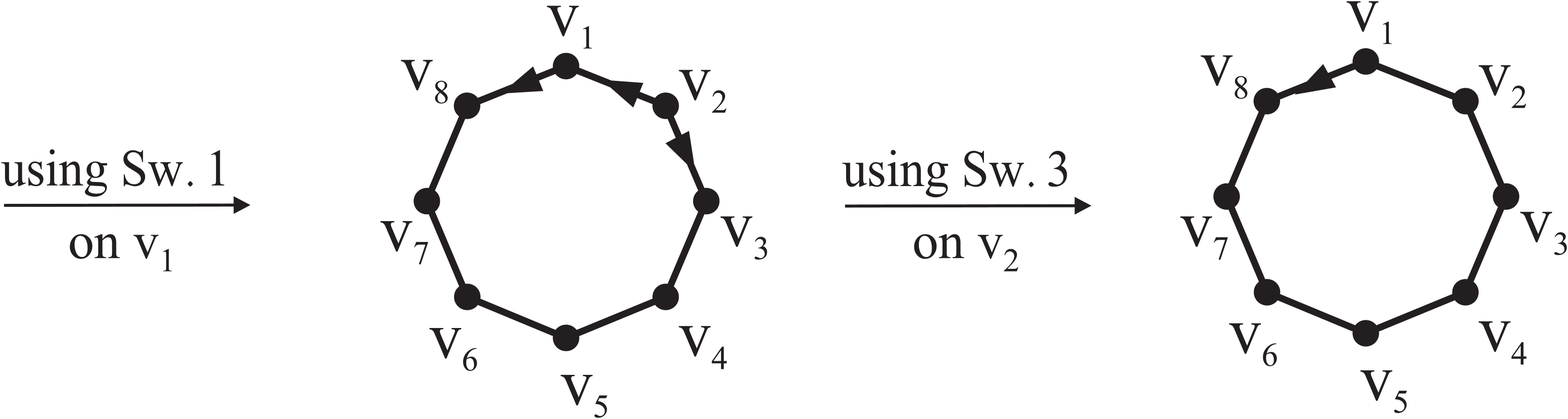}
	\caption{Switching a mixed cycle by the four switching functions introduced above.}
		\label{switch}
\end{figure}
\newpage
Now, we determine the spectrum of $ C_{n}$, $ C^{1}_{n} $ and $ C^{2}_{n} $. The spectrum of $ C_{n} $ is well-known \cite[p.3]{cve}:
\begin{equation*}
\text{Spec}_{H}(C_{n})= \{2\cos \frac{2k}{n}\pi , ~ k=0,1, \ldots , n-1 \}.
\end{equation*}
To determine the spectrum of $ C^{2}_{n} $, without loss of generality, assume that the directed edges of $ C^{2}_{n} $ are $ (v_{1},v_{2}) $ and $ (v_{2},v_{3}) $. Define the switching function $\theta$,
\begin{equation*}
\theta(v_{2})=i \text{~~and~~} \theta(x) = 1,\qquad\text{~~for every~~} x \in V(C_{n}^{2})\backslash\{v_{2}\}.
\end{equation*} 
Hence, if $ H $ is the Hermitian adjacency matrix of $C^{2}_{n}$, then it is similar to the signed adjacency matrix of $ (C_{n},-) $. So, $ \text{Spec}_{H}(C^{2}_{n})=\text{Spec}(C_{n},-) $. Now, by \cite{bel}, we have
\begin{equation*}
\text{Spec}_{H}(C^{2}_{n})= \{2\cos \frac{2k+1}{n}\pi , ~ k=0,1, \ldots , n-1 \}.
\end{equation*}
Finally, we turn to determine the spectrum of $ C^{1}_{n} $.
\begin{theorem}
	\label{spectype1}
	For every integer $n\geq3$, the following holds:
	\begin{equation}\label{et2}
\phi((C_{2n},-),\lambda)= (\phi(C^{1}_{n}, \lambda))^{2}.
	\end{equation}
\end{theorem}
\begin{proof}
	By Equations \eqref{eccpm} and \eqref{eza} and the fact that the number of $ k $-matchings of $ C_{n} $ is $\displaystyle \dfrac{n}{n-k} \binom{n-k}{k}$ \cite[p.14]{god}, we have
\begin{align*}
\phi(C^{1}_{n}, \lambda) &= \sum_{j=0}^{\lfloor \frac{n}{2} \rfloor} (-1)^{j}\frac{n}{n-j} \binom{n-j}{j}\lambda^{n-2j}, \\
\phi((C_{2n},-),\lambda) &= \sum_{j=0}^{n} (-1)^{j}\frac{2n}{2n-j} \binom{2n-j}{j}\lambda^{2n-2j}+2.
\end{align*}
To prove Equation \eqref{et2}, we show that for $k=0,1,\ldots,n$, the coefficient of $ \lambda^{2n-2k} $ in both sides of \eqref{et2} are the same. This coefficient in $(\phi(C^{1}_{n}, \lambda))^{2}$ is as follows:
\begin{equation*}
\sum_{j=0}^{k} (-1)^{j}\frac{n}{n-j} \binom{n-j}{j}(-1)^{k-j}\frac{n}{n-(k-j)} \binom{n-(k-j)}{k-j}=
\end{equation*}
\begin{equation*}
(-1)^{k}\sum_{j=0}^{k}\frac{n}{n-j} \binom{n-j}{j}\frac{n}{n-(k-j)} \binom{n-(k-j)}{k-j}.
\end{equation*}
On the other hand, the coefficient of $ \lambda^{2n-2k} $ in $\phi((C_{2n},-),\lambda)$ is
\begin{equation*}
(-1)^{k}\frac{2n}{2n-k}\binom{2n-k}{k}.
\end{equation*}
The number of $ k $-matchings in the disjoint union of two copies of $ C_{n} $ and the number of $ k $-matchings of $ C_{2n} $ is $\displaystyle\sum_{j=0}^{k}\frac{n}{n-j} \binom{n-j}{j}\frac{n}{n-(k-j)} \binom{n-(k-j)}{k-j}$ and $\displaystyle\frac{2n}{2n-k}\binom{2n-k}{k}$, respectively. Therefore, it is enough to prove that these two numbers are the same. Indeed, we would like to define a bijection between $k$-matchings of $C_{2n}$ and $k$-matchings of the disjoint union of two copies of $C_{n}$. First, let $ k < n $, and label the vertices of $C_{2n}$ by 
$ \{1,2,\ldots,n, 1^{\prime},2^{\prime}, \ldots ,n^{\prime}\}$ in clockwise order and let $M$ be a matching of size $k$ in $C_{2n}$. Since $M$ is not a perfect matching, there exists $i$
such that none of the edges $i(i+1)$ and $i^{\prime}((i+1)^{\prime})$ is in $M$ (Note that if $i=n$, then $i+1=1^{\prime}$ and $(i+1)^{\prime}=1$). Choose the smallest $i$ with this property and replace the edges $i(i+1)$ and $i^{\prime}((i+1)^{\prime})$ with the edges $i((i+1)^{\prime})$
and $i^{\prime}(i+1)$ in $C_{2n}$ to make two disjoint copies of $C_{n}$. Now, drop prime from the vertices $ (i+1)^{\prime}, \ldots , n^{\prime} $ and add prime to the vertices $ i+1, \ldots , n $. This creates a $k$-matching of two disjoint copies of $C_{n}$. We show that this procedure is reversible. Let $ M $ be a $ k $-matching in the disjoint union of two copies of $ C_{n} $. Label the vertices of one of the two $ C_{n} $ by $ 1,\ldots , n $ and the other by $ 1^{\prime}, \ldots , n^{\prime} $. Since $ k<n $, the restriction of $ M $ to at least one of the $ C_{n} $ is not a perfect matching. Consider the smallest $ i $ such that none of the edges $i(i+1)$ and $i^{\prime}(i+1)^{\prime}$ is in $M$ and replace the edges $i(i+1)$ and $i^{\prime}((i+1)^{\prime})$ with the edges $i((i+1)^{\prime})$
and $i^{\prime}(i+1)$ to make $C_{2n}$. Relabel the vertices $ (i+1)^{\prime}, \ldots , n^{\prime} $ with $ i+1, \ldots , n $  and the vertices $ i+1, \ldots , n $ with $ (i+1)^{\prime}, \ldots , n^{\prime} $. This leads to a $ k $-matching of $ C_{2n} $ which is obviously the reverse of the previous procedure.

Now, if $ k=n $, then two cases should be considered. If $ n $ is odd, then one can easily see that the constant term of both $ \phi((C_{2n},-),\lambda) $ and $\phi(C^{1}_{n}, \lambda)$ is zero. If $ n $ is even, then the constant term of $ \phi((C_{2n},-),\lambda) $ is 4 and the constant term of $\phi(C^{1}_{n}, \lambda)$ is 2. This completes the proof.
\end{proof}
It is easy to see that the multiplicity of each eigenvalue of $ (C_{2n},-) $ is 2. So, we have the following corollary:
\begin{corollary}
	For every integer $n\geq3$, the following holds:
	\begin{equation*}
	\text{Spec}_{H}(C^{1}_{n})= \{2\cos \frac{2k+1}{2n}\pi , ~ k=0,1, \ldots , n-1 \}.
	\end{equation*}
\end{corollary}
\subsection{Mixed Unicyclic Graphs}
Now, similar to Theorem \ref{Types}, a result regarding mixed unicyclic graphs is obtained. Using Theorems \ref{forest} and \ref{Types}, we have the following theorem.
\begin{theorem}
	\label{prop1}
	Let $ X $ be a mixed unicyclic graph whose cycle is of order $ n $. Then $ X $ is switching equivalent to a mixed unicyclic graph $ Y $, where $G(Y)=G(X)$ and $ Y $ is one of the following mixed graphs:
	\begin{itemize}
		\item[(i)] All edges of $ Y $ are undirected.
		\item[(ii)] The cycle of $ Y $ is $ C^{1}_{n} $ and all other edges are undirected.
		\item[(iii)] The cycle of $ Y $ is $ C^{2}_{n} $ and all other edges are undirected.
	\end{itemize} 
	We call them mixed unicyclic graphs of Type $0$, $1$ and $2$, respectively.
\end{theorem}
\begin{remark}
	\label{remark8}
	To determine the spectrum of $ C_{n}^{2} $, we use a switching function which switches $ C_{n}^{2} $ to $ (C_{n},-) $. Now, if $ X $ is a mixed unicyclic graph of Type $2$ and $ G(X)=Y $, then this switching function switches $ X $ to the signed graph $ (Y,-) $.
\end{remark}

\begin{remark}	
In the next section, two mixed unicyclic graphs $ G_{t} $ and $ G_{t}^{t+m} $ appear, where $ G_{t} $ is the union of $C_{4}^{2}$ and $P_{t}$, such that one endpoint of $P_{t}$ is joined to a vertex of $C_{4}^{2}$, and $ G_{t}^{t+m} $ is the union of $C_{4}^{2}$, $P_{t}$ and $P_{t+m}$, such that one endpoint of $P_{t}$ is joined to one vertex of $C_{4}^{2}$, and one endpoint of $P_{t+m}$ is joined to the opposite vertex of $C_{4}^{2}$. Now, by Remark \ref{remark8}, we switch $ G_{t} $ and $ G_{t}^{t+m} $ to the signed graphs $ H_{t} $ and $ H_{t}^{t+m} $, respectively (See Figure \ref{gh}). Their spectrums were determined in \cite{sop}.
$$\textrm{Spec}_{H}(G_{t}) = \textrm{Spec}(H_{t}) =\{2cos\dfrac{2k+1}{2t+4}\pi,k=0,\ldots,t+1\}\cup \{2\cos\dfrac{\pi}{4},2\cos\dfrac{3\pi}{4}\},$$
$$\textrm{Spec}_{H}(G^{t+m}_{t})=\textrm{Spec}(H_{t}^{t+m})=\{2cos\dfrac{2k+1}{2t+2m+4}\pi,k=0,\ldots,t+m+1\}\cup$$ $$\{2cos\dfrac{2k+1}{2t+4}\pi,k=0,\ldots,t+1\}.$$
\end{remark}
\begin{figure}[!h]
	\centering
	\includegraphics[height=3.5cm]{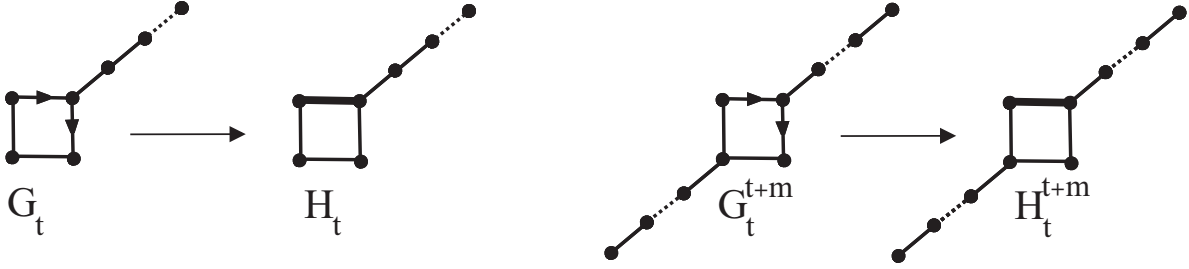}
	\caption{Two mixed graphs $G_t$ and $G_{t}^{t+m}$ can be switched to the signed graphs $H_t$ and $H_{t}^{t+m}$, respectively. The thick edges are negative.}
	\label{gh}
\end{figure}
\section{Mixed Cospectral Mates of Paths}
In this section we deal with the problem of determining the mixed cospectral mates of paths. First, we state two following theorems which are important tools in proving our results.\\
Analogous to simple graphs \cite{sch}, one can obtain a Schwenk-like formula for the mixed graphs. The following theorem expresses this formula.
\begin{theorem}\label{tsl}
	Let $ X $ be a mixed graph, where $ u \in V(X) $. Then
	\begin{itemize}
		\item[(i)] $\phi(X, \lambda)= \lambda \phi(X\setminus u, \lambda)- \sum \limits_{u \sim v} \phi (X\setminus \{u,v\},\lambda)-2 \sum \limits_{Z \in C(u)}\phi(X \setminus V(Z),\lambda) h(Z)$,
		
		where $ C(u) $ is the set of all real cycles passing through $ u $ and $h(Z)$ is the value of $Z$.
		\item[(ii)]
		$\phi(X, \lambda)= \phi(X\setminus e)- \phi (X\setminus \{u,v\},\lambda)-2 \sum \limits_{Z \in C(e)}\phi(X \setminus V(Z),\lambda) h(Z)$,
		
		where $ C(e) $ is the set of all real cycles containing the edge $ e=uv $.
	\end{itemize}
\end{theorem}
The following theorem states the interlacing theorem for mixed graphs.
\begin{theorem}\cite{guom}\label{interlace}
	The eigenvalues of an induced subgraph of a mixed graph interlace the eigenvalues of the mixed graph.
\end{theorem}
Now, we are ready to determine the mixed cospectral mates of paths. Recall that \cite[p.47]{cve}
\begin{equation*}
\text{Spec}_{H}(P_{n})= \{2\cos \dfrac{k\pi}{n+1} , ~ k=1, \ldots , n \}.
\end{equation*}
 In all following lemmas and theorems, $X$ is a mixed cospectral mate of $P_{n}$, so it is of order and size $n$ and $n-1$, respectively. Since as a simple graph, $P_{n}$ is determined by its spectrum, by Theorem \ref{forest}, $X$ has at least two connected components. Furthermore, all eigenvalues of $X$ are simple and in $(-2,2)$. Throughout this section, all graphs we refer to are in Figures \ref{unacceptable} and \ref{acceptable}, see Appendix A.
 
 By Theorem \ref{interlace}, we have the following lemma.
\begin{lemma}
	\label{induced}
	For every odd positive integer $k$, $X$ has no cycles of Type 0 and $C_{k}^{2}$ as an induced subgraph.
\end{lemma}
Next, we obtain an upper bound on the maximum degree of $X$.
\begin{lemma}\label{max3}
	The maximum degree of $X$ is at most 3.
\end{lemma}
\begin{proof}
	By examining all mixed graphs on five vertices with maximum degree $4$ using a computer search, we see that all such mixed graphs are $(-2,2)$-out, and therefore, by Theorem \ref{interlace}, these mixed graphs cannot be an induced subgraph of $X$, and the proof is complete. 
\end{proof}
The next two lemmas characterize all possible induced mixed $\theta$-graphs of $X$.
\begin{lemma}
	\label{theta}
	The mixed graph $X$ has no induced mixed proper $\theta$-graphs.
\end{lemma}
\begin{proof}
	By contradiction, assume that there is an induced subgraph of $X$ whose underlying graph is $\theta_{p,q,r}$, where $p,q,r\geq3$. If $p,q,r\geq4$, then the Graph (a) is an induced subgraph of $X$, a contradiction. So, we can assume that $p=3$ and $3\leq q\leq r$. If $q\geq 4$ and $r\geq 7$, then the Graph (b) is an induced subgraph of $X$, a contradiction. Also, if $r\geq q\geq 5$, then the Graph (c) is an induced subgraph of $X$, a contradiction. The cases where $q=4$ and $4\leq r\leq 6$ can be investigated using a computer search, and the result is that they are all $(-2,2)$-out, a contradiction. So we have $p=q=3$ and $r\geq 3$. One can see that all mixed graphs with the underlying graph $\theta_{3,3,r}$, except Graphs $Y_{1}$ and $Y_{2}$ in Figure \ref{y_1 y_2} can be switched to a mixed graph having either a mixed cycle of Type 0 or an odd mixed cycle of Type 2 as an induced subgraph and therefore are $(-2,2)$-out. By Theorem \ref{tsl}, the following hold:
	$$\phi(Y_{1},\lambda)=\lambda\phi(D_{r+1},\lambda)-2\phi(P_{r},\lambda)-\phi(D_{r},\lambda)+2\lambda,$$
	$$\phi(Y_{2},\lambda)=\lambda\phi(D_{r+1},\lambda)-2\phi(P_{r},\lambda)-\phi(D_{r},\lambda)+2\phi(P_{r-2},\lambda).$$
	Also, by Theorem \ref{tsl}, we have $\phi(P_{k},\lambda)=\lambda\phi(P_{k-1},\lambda)-\phi(P_{k-2},\lambda)$. By induction on $ k $, one can easily prove that we have $\phi(P_{k},2)=k+1$. Moreover, by Theorem \ref{tsl}, we find that $\phi(D_{k},\lambda)=\lambda\phi(P_{k-1},\lambda)-\lambda\phi(P_{k-3},\lambda)$, which yields that $\phi(D_{k},2)=4$. Therefore, we have $\phi(Y_{1},2)=6-2r$ and $\phi(Y_{2},2)=0$. Since $r\geq 3$, we have $6-2r\leq0$. On the other hand, $\phi(Y_{1},\lambda)$ and $\phi(Y_{2},\lambda)$ are both monic polynomials, so, they both have a root which is at least $2$, a contradiction. The proof is complete.
			\begin{figure}[!h]
				\centering
				\includegraphics[height=2.2cm]{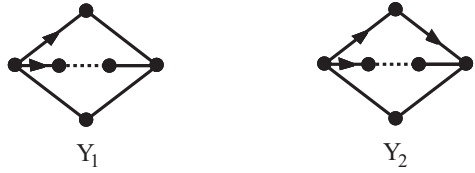}
				\caption{The Graphs $Y_1$ and $Y_2$ are $(-2,2)$-out.}
				\label{y_1 y_2}
			\end{figure}
\end{proof}
\begin{lemma}
	\label{2-theta}
	If $X$ has an induced subgraph with the underlying graph $\theta_{2,q,r}$, where $3\leq q\leq r$, then either $(q,r)=(4,4)$ or $(q,r)=(4,6)$.
\end{lemma}
\begin{proof}
	First, let $q\geq6$. So, the Graph (c) is an induced subgraph of $X$, a contradiction. Also, if $q\geq5$ and $r\geq8$, then the Graph (b) is an induced subgraph of $X$, a contradiction. Moreover, if $q=4$ and $r\geq7$, then $X$ has one of the Graphs (l) of order 5 or (j) as an induced subgraph, a contradiction. Finally, if $q=3$ and $r\geq5$, then $X$ has the Graph (f) as an induced subgraph, a contradiction. Now, by a computer search, one can see that a mixed graph with one of the underlying graphs $\theta_{2,5,5}$, $\theta_{2,5,6}$, $\theta_{2,5,7}$, $\theta_{2,4,5}$, $\theta_{2,3,3}$ or $\theta_{2,3,4}$ is $(-2,2)$-out, and the proof is complete.
\end{proof}
Now, we investigate the induced mixed cycles of $X$.
\begin{lemma}
	\label{proper}
	No component of $X$ has a mixed cycle of order at least $8$ as a proper induced subgraph.
\end{lemma}
\begin{proof}
	By contradiction, assume that a mixed cycle of order at least $8$, say $C$, is a proper induced subgraph of a component of $X$, say $H$. By Lemma \ref{induced}, $C$ is of Type $1$ or $2$. There is a vertex $u$ in $H$ which is not in $C$ and is adjacent to at least one vertex in $C$. By Lemma \ref{max3}, $u$ has degree at most 3. If $u$ is adjacent to exactly one vertex in $C$, then $H$ has the Graph (c) as an induced subgraph, a contradiction. Assume that $u$ is adjacent to exactly two vertices in $C$, say $v$ and $w$. If $v$ and $w$ are not adjacent, then $H$ has an induced subgraph whose underlying graph is a proper $\theta$-graph, which contradicts Lemma \ref{theta}. So, $v$ and $w$ are adjacent, and by Lemma \ref{induced}, the induced cycle on the vertices $u$, $v$ and $w$ is of Type $1$. However, since $C$ has at least $8$ vertices, $H$ has the Graph (d) as an induced subgraph, a contradiction. The only case left to examine is when $u$ is adjacent to three vertices in $C$. If no two of these three vertices are adjacent, then one can see that $H$ has one of the family of the Graph (e) as an induced subgraph, a contradiction. So, at least two of these three vertices are adjacent. Therefore, $H$ has the Graph (d) as an induced subgraph, a contradiction, and the proof is complete.
\end{proof}
\begin{lemma}\label{8induced-type2}
	Let $k\geq8$ be an integer. Then, $C_{k}^{2}$ is not an induced subgraph of $X$.
\end{lemma}
\begin{proof}
	If $k$ is odd, then by Lemma \ref{induced}, we are done. So, let $k\geq8$ be an even integer. By contradiction, suppose that $C_{k}^{2}$ is an induced subgraph of $X$. We know that $C_{k}^{2}$ has eigenvalues of multiplicity $2$, so it cannot be a component of $X$. By Lemma \ref{proper}, $C_{k}^{2}$ cannot be a proper induced subgraph of a component of $X$, and the proof is complete.
\end{proof}
\begin{lemma}
	If $C_{6}^{2}$ is an induced subgraph of a component of $X$, say $H$, then $H$ is one of the mixed Graphs (g) or (h).
\end{lemma}
\begin{proof}
	Since $C_{6}^{2}$ has eigenvalues of multiplicity 2, we have $H\not=C_{6}^{2}$. We show that each vertex in $H$ is adjacent to at least one vertex in $C_{6}^{2}$. By contradiction, assume that there is a vertex $v$ in $H$ which has distance 2 from $C_{6}^{2}$, and let $u$ be the vertex adjacent to $v$ and a vertex of $C_{6}^{2}$. If $u$ is adjacent to exactly one vertex in $C_{6}^{2}$, then the Graph (a) is an induced subgraph of $H$, a contradiction. So we can assume that $u$ is adjacent to exactly two vertices of $C_{6}^{2}$, say $z$ and $w$. If $z$ and $w$ are not adjacent, then $H$ has an induced mixed proper $\theta$-graph, which contradicts Lemma \ref{theta}. So, $z$ and $w$ are adjacent, and by Lemma \ref{induced}, the induced cycle on the vertices $u$, $z$ and $w$ is of Type $1$, and $H$ has the Graph (f) as an induced subgraph, a contradiction.
	
	So, $H$ has order at most $12$. Let $v$ be a vertex in $V(H)\backslash V(C_{6}^{2})$. If $v$ is adjacent to exactly two vertices of $C_{6}^{2}$, then $H$ has an induced subgraph with the underlying graph $\theta_{3,p,q}$, where $p,q\geq2$, which contradicts Lemmas \ref{theta} and \ref{2-theta}. If $v$ is adjacent to exactly three vertices of $C_{6}^{2}$, then none of these three vertices are adjacent, because otherwise $H$ has an induced subgraph with the underlying graph $\theta_{2,3,p}$, where $p\geq3$, a contradiction. So, a mixed graph with the underlying Graph (m) is an induced subgraph of $H$, a contradiction. Therefore, $v$ is adjacent to exactly one vertex of $C_{6}^{2}$. If $|V(H)|=7$, then $H$ is the Graph (g). Now, suppose that $|V(H)|\geq8$ and let $v_{1}$ and $v_{2}$ be two vertices in $V(H)\backslash V(C_{6}^{2})$. For $i=1,2$, let $u_{i}$ be the vertex in $C_{6}^{2}$ adjacent to $v_{i}$. If $v_{1}$ and $v_{2}$ are adjacent, then $u_{1}$ and $u_{2}$ are also adjacent, because otherwise a mixed proper $\theta$-graph is an induced subgraph of $H$, a contradiction. Now, the induced mixed cycle on the vertices $u_{1}$, $u_{2}$, $v_{1}$ and $v_{2}$ is not of Type 1, because otherwise the Graph (l) of order 5 is an induced subgraph of $H$, a contradiction. So, by Lemma \ref{induced}, this cycle is of Type 2. If $v_{1}$ and $v_{2}$ are not adjacent, then $u_{1}$ has distance 3 from $u_{2}$, because otherwise one of the family of the Graph (e) is an induced subgraph of $H$, a contradiction. Therefore, we have $|V(H)|=8$. Since the Graph ($\text{g}_{1}$) in Figure \ref{g_1} has eigenvalues of multiplicity 2,  $H$ is switching equivalent to the Graph (h), and the proof is complete. 	
\end{proof}
\begin{figure}[!h]
	\centering
	\includegraphics[height=2cm]{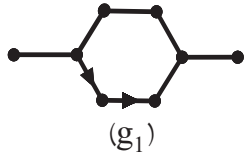}
	\caption{The Graph $(\text{g}_1)$ has eigenvalues of multiplicity 2.}
	\label{g_1}
\end{figure}
\begin{lemma}\label{induced-type1}
	If $C_{j}^{1}$, where $j\geq3$, is an induced subgraph of a component of $X$, say $H$, then either $H=C_{j}^{1}$ or $H$ is the Graph (k).
\end{lemma}
\begin{proof}
	By contradiction, assume that $H$ is neither $C_{j}^{1}$ nor the Graph (k). By Lemma \ref{proper}, we have $j\leq7$. Since $H\not=C_{j}^{1}$, there is a vertex $v$ in $V(H)\backslash V(C_{j}^{1})$ adjacent to a vertex of $C_{j}^{1}$. First, assume that $v$ is adjacent to exactly one vertex in $C_{j}^{1}$. If $ j\geq 4 $, then $ H $ has the Graph (l) of order $ j+1 $ as an induced subgraph, a contradiction. Therefore, we have $j=3$. So, the Graph (k) is an induced proper subgraph of $H$. Therefore, there exists an induced subgraph of $H$ of order 5, say $Y$, in which a vertex $u$ is adjacent to at least one of the vertices of the induced subgraph (k). It can be seen that $Y$ has one of the Graphs (d), (f) or a mixed graph with the underlying graph $\theta_{2,3,p}$, where $p\geq3$, as an induced subgraph, a contradiction.
	
	If $v$ is adjacent to exactly two vertices in $C_{j}^{1}$, then $H$ has a mixed graph with the underlying graph $\theta_{3,q,r}$, where $q,r\geq2$, as an induced subgraph, which contradicts Lemmas \ref{theta} and \ref{2-theta}. Therefore, $v$ is adjacent to three vertices in $C_{j}^{1}$. If $j=3$, then $H$ has an induced subgraph with the underlying graph $K_{4}$. Using a computer search, we see that all such mixed graphs are $(-2,2)$-out, a contradiction. So, $j\geq4$. If two of the three vertices adjacent to $v$ are adjacent, then $H$ has an induced subgraph with the underlying graph $\theta_{2,3,p}$, where $p\geq3$, a contradiction. Therefore, we have $j\geq6$. If $j=6$, then $H$ has an induced subgraph with the underlying Graph (m), a contradiction. So, $j=7$ and $H$ has an induced subgraph with the underlying graph $\theta_{2,4,5}$, a contradiction. The proof is complete.
\end{proof}
\begin{lemma}
	Let $H$ be a component of $X$. If $C_{4}^{2}$ is an induced subgraph of $H$ and $C_{6}^{2}$ is not, then $H$ is one of the Graphs (o), (p), (q), (r), (s), (t), (u), (v), ($ \text{G}_{t} $) or ($ \text{G}_{t}^{t+m} $).
\end{lemma}
\begin{proof}
	By Lemmas \ref{induced} and \ref{induced-type1}, we know that all induced cycles of $H$ are of Type 2 and even order. Moreover, by Lemma \ref{8induced-type2}, all induced cycles of $H$ are $C_{4}^{2}$. Since $C_{4}^{2}$ has no simple eigenvalue, we have $H\not=C_{4}^{2}$. Let $A$ be the set of vertices in $V(H)\backslash V(C_{4}^{2})$ which are adjacent to at least one vertex in $C_{4}^{2}$. We have $1\leq|A|\leq4$, and each vertex in $A$ is adjacent to exactly one vertex in $C_{4}^{2}$, because otherwise a mixed cycle of order 3 or a mixed graph with the underlying graph $\theta_{3,3,3}$ is an induced subgraph of $H$, a contradiction. In the induced graph $\langle A \rangle$, no two edges are adjacent, because otherwise $H$ contains an induced cycle of order 3, 5 or 6, a contradiction. Now, assume that $\langle A \rangle $ has two disjoint edges. Then, either a mixed graph with the underlying Graph (n) or a mixed cycle of order 5 is an induced subgraph of $H$, a contradiction. Therefore, suppose that $\langle A\rangle$ has at most one edge. Assume that $\langle A \rangle $ has no edge. We show that all vertices in $V(H)\backslash V(C_{4}^{2})$ have degree at most 2. By contradiction, let $v$ be a vertex in $V(H)\backslash V(C_{4}^{2})$ of degree 3 which has minimum distance from $C_{4}^{2}$. If $v\in A$, then since $\langle A\rangle$ has no edge and $C_{4}^2$ is the only induced cycle of $H$, the Graph (e) of order 6 is an induced subgraph of $H$, a contradiction. Similarly, if $v\not\in A$, we reach a contradiction. Since $H$ has no induced cycle on more than 4 vertices, $H$ is a mixed unicyclic graph. So, the Graphs (o), (p), (q), (r), (s), ($ \text{G}_{t} $), ($ \text{G}_{t}^{t+m} $) and the Graph ($\text{j}_{3}$) in Figure \ref{j3} are the only possibilities for $H$, because otherwise one of the Graphs (a), (c), (j), ($\text{j}_{1}$) or ($\text{j}_{2}$) is an induced subgraph of $H$, a contradiction. Moreover, the Graph ($\text{j}_{3}$) has an eigenvalue of multiplicity 2, so $H$ is not the Graph ($\text{j}_{3}$).
	
	Finally, we can assume that $\langle A \rangle $ has exactly one edge. Let $v_{1}$ and $v_{2}$ be two adjacent vertices in $A$. For $i=1,2$, let $u_{i}$ be the vertex in $C_{4}^{2}$ adjacent to $v_{i}$. Then, $u_{1}$ and $u_{2}$ are adjacent, because otherwise $H$ has an induced mixed cycle of order 5, a contradiction. Now, we show that all vertices in $V(H)\backslash (V(C_{4}^{2})\cup\{v_{1},v_{2}\})$ have degree at most 2. By contradiction, assume there exists a vertex in $V(H)\backslash (V(C_{4}^{2})\cup\{v_{1},v_{2}\})$ of degree 3. Since $\langle A\rangle$ has exactly one edge, $C_{4}^2$ is the only induced cycle of $H$ and a mixed graph with the underlying Graph (n) is $(-2,2)$-out, one can see that one of the family of the Graph (e) is an induced subgraph of $H$, a contradiction. Also, since a mixed graph with the underlying Graph (n) is $(-2,2)$-out and $C_{4}^2$ is the only induced cycle of $H$, $ H $ is the Graph ($ \text{t}_{5} $) in Figure \ref{j3}. Now, since the Graphs (e), ($\text{t}_{1}$), ($\text{t}_{2}$) and ($\text{t}_{3}$) are $(-2,2)$-out, $H$ can only be one of the Graphs (t), (u), (v) or the Graph ($\text{t}_{4}$) in Figure \ref{j3}. On the other hand, the Graph ($\text{t}_{4}$) has an eigenvalue of multiplicity 2, so $H$ is not the Graph ($\text{t}_{4}$), and the proof is complete.
		\begin{figure}
			\centering
			\includegraphics[height=2cm]{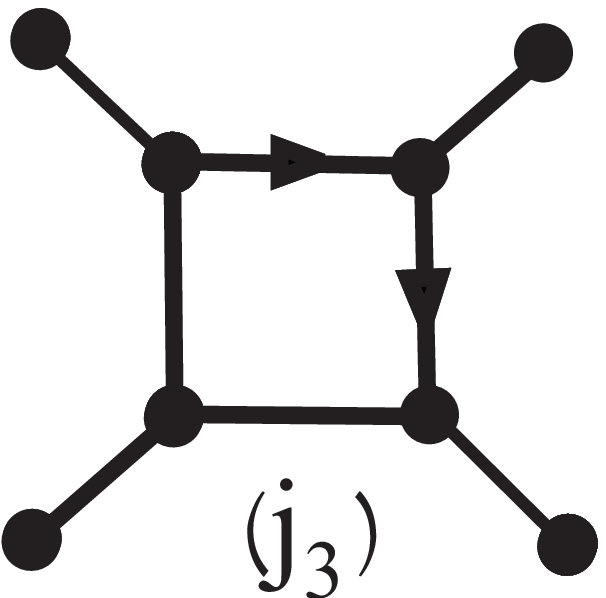}
			\qquad\qquad\qquad			\includegraphics[height=1.8cm]{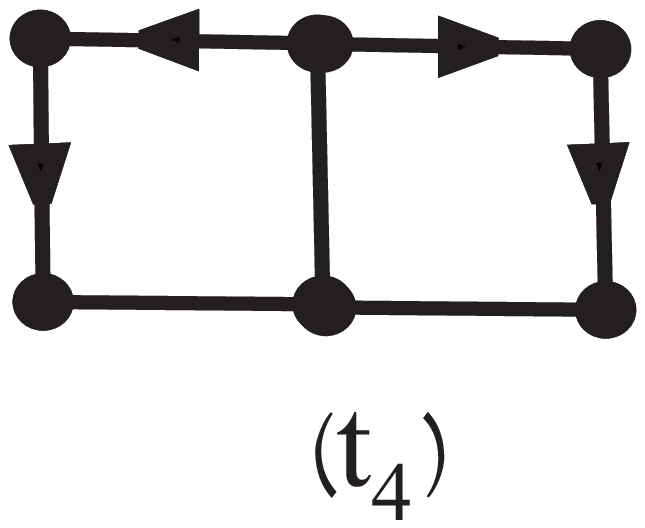}
			$$\includegraphics[height=2cm]{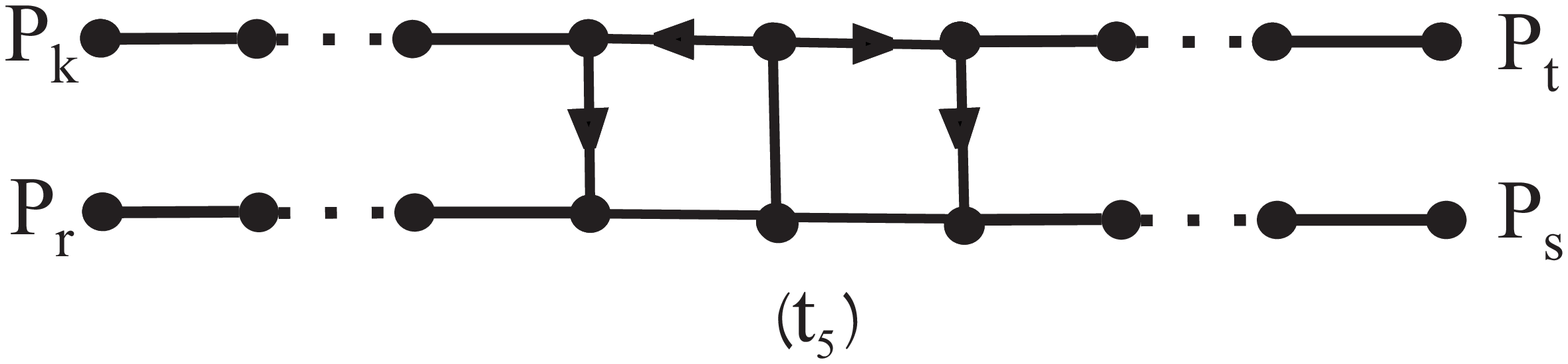}$$
			\caption{The Graphs $(\text{j}_{3})$, $ (\text{t}_{4}) $ and $ (\text{t}_{5}) $.}
			\label{j3}
		\end{figure}
\end{proof}
Now, using these lemmas, we list all possible components of $X$ in Figure \ref{acceptable}.
\subsection{Paths of Even Orders}
If $X$ is a mixed cospectral mate of $P_{n}$, where $n$ is an even integer, then by the spectrum of the admissible graphs in Appendix A, the Graphs (g), (p), (r), (t), ($\text{D}_{t}$) and (w) cannot be a component of $X$ since they all have a zero eigenvalue. Furthermore, the Graphs (q), (s), (k), (v), ($\text{G}_{t}$), ($\text{G}_{t}^{t+m}$), (y), (z) and $C_{t}^{1}$ all have an eigenvalue of the form $2\cos(\dfrac{p}{q}\pi)$, where $\gcd(p,q)=1$ and $q$ is even, which cannot be an eigenvalue of $X$. Therefore, only the Graphs (o), (u), (h) and even paths can be components of $X$. Moreover, two Graphs (h) and (u) are cospectral, so at most one of them can be a component of $X$. We remind that $X$ has at least two components. Now, we examine all possible cases for the number of components of $X$.
\begin{lemma}
	\label{2comp}
	If $X$ has two components, then $n=8$, and $X$ is the disjoint union of $P_{2}$ and the Graph (o).
\end{lemma}
\begin{proof}
	Since $X$ has exactly two components and it has order $n$ and size $n-1$, these components should be $P_{n-6}$ and the Graph (o). Obviously, $\lambda_{1}(P_{n-6})<\lambda_{1}(P_{n})$, so $\lambda_{1}(P_{n})$ should be the largest eigenvalue of the Graph (o). Therefore, we have $2\cos\dfrac{\pi}{9}=2\cos\dfrac{\pi}{n+1}$. This implies that $n=8$, and one can easily see that the disjoint union of $P_{2}$ and the Graph (o) is a cospectral mate of $P_{8}$. The proof is complete.
\end{proof}
\begin{lemma}
	If $X$ has exactly three components, then $n=14$, and $X$ is the disjoint union of either $P_{2}$, $P_{4}$ and the Graph (u), or $P_{2}$, $P_{4}$ and the Graph (h).
\end{lemma}
\begin{proof}
	Since $X$ has three components, one can see that these components are either two even paths and the Graph (h), or two even paths and the Graph (u). In each case, the largest eigenvalue of $X$ cannot be an eigenvalue of a shorter path, so $2\cos\dfrac{\pi}{15}=2\cos\dfrac{\pi}{n+1}$. Therefore, $n=14$, and one can see that the disjoint union of either $P_{2}$, $P_{4}$ and the Graph (u), or $P_{2}$, $P_{4}$ and the Graph (h) is a cospectral mate of $P_{14}$.
\end{proof}
Finally, we obtain an upper bound for the number of components of $X$.
\begin{lemma}
	The mixed graph $X$ has at most three components.
\end{lemma}
\begin{proof}
	By contradiction, assume that $X$ has more than three components. Since $X$ has size $n-1$, it is the disjoint union of either two even paths and two Graphs (o) and (h), or two even paths and two Graphs (o) and (u). In each case, since the Graphs (h) and (u) have an eigenvalue larger than the maximum eigenvalue of (o), we conclude that $2\cos\dfrac{\pi}{15}=2\cos\dfrac{\pi}{n+1}$, and therefore $n=14$. Because of the order of the Graphs (o), (h) and (u), $X$ has no path components, a contradiction.
\end{proof}
Now, the following theorem is an immediate consequence of three previous lemmas.
\begin{theorem}\label{finalevenorder}
	Let $n$ be an even positive integer. Then $P_{n}$ is DHS if and only if $n\not\in\{8,14\}$.
\end{theorem}
\newpage
\subsection{Paths of Odd Orders}
We consider two cases:
\begin{enumerate}
	\item [Case 1.] $n\equiv1\pmod4$:
	
	One can see that $P_{4k+1}$ is a cospectral mate of the disjoint union of $P_{2k}$ and $C_{2k+1}^{1}$, where $k$ is a positive integer. So, we have the following result.
	\begin{theorem}
		For every positive integer $k$, $P_{4k+1}$ is not DHS.
	\end{theorem}
	
	Now, we express all mixed cospectral mates of the path $ P_{4k+1} $. Recall that
	\begin{equation*}
	\text{Spec}_{H}(P_{4k+1})=\{2 \cos \dfrac{j}{4k+2}\pi; \quad j=1,2, \ldots , 4k+1\}.
	\end{equation*}
	Let $ X $ be a mixed cospectral mate of $ P_{4k+1} $. So, the eigenvalues of $ X $ are of the form $ 2 \cos(\dfrac{p}{q}\pi) $, where $ \gcd(p,q)=1 $ and $ 4 \nmid q $. Now, by the spectrum of the graphs given in Figure \ref{acceptable}, one can find that the only possible components of $X$ are the Graphs (o), (p), (t), (u), (v), (h), ($ \text{G}_{t}^{t+m} $), (w), (z), ($\text{D}_{l} $), paths and odd cycles of Type 1, where $m$ and $l$ are even and $t$ is odd. Note that if $m$ and $l$ are even and $t$ is odd, then two Graphs ($ \text{G}_{t}^{t+m} $) and ($ \text{D}_{l} $) have a zero eigenvalue of multiplicity 2. So, these two graphs cannot be a component of $ X $.
	
	Among the Graphs (o), (p), (t), (u), (v), (h), (w) and (z), the largest eigenvalue is $2\cos\dfrac{\pi}{30}$. Thus, if $ 4k+1 > 29 $, then $ C^{1}_{2k+1} $ is a component of $X$. Since $ 2k \geq 16 $, by Theorem \ref{finalevenorder}, $P_{2k}$ is DHS. Therefore, since $P_{2k}\cup C_{2k+1}^{1}$ is a cospectral mate of $P_{4k+1}$, we can state the following theorem:
	\begin{theorem}
		If $ k\geq8  $, then the only mixed cospectral mate of $ P_{4k+1} $ is $ P_{2k} \cup C^{1}_{2k+1}$.
	\end{theorem} 
	Now, by the spectrum of the graphs given in Figure \ref{acceptable}, the class of all mixed cospectral mates of $P_{4k+1}$ is as follows:
	
		\{$ P_{5},C^{1}_{3} \cup P_{2}$\},\quad\{$ P_{9} , C^{1}_{5} \cup P_{4} $\},\quad\{$ P_{13} , C^{1}_{7} \cup P_{6} , (\text{p}) \cup P_{6} $\},\\ \{$ P_{17} , C^{1}_{9} \cup P_{8} , C^{1}_{9} \cup P_{2} \cup (\text{o}), (\text{v}) \cup P_{8} \cup P_{1}, (\text{v}) \cup (\text{o}) \cup P_{2} \cup P_{1} $\},\\ \{$ P_{21} , C^{1}_{11} \cup P_{10} $\},\quad \{$ P_{25} ,C^{1}_{13} \cup P_{12} $\},\\ \{$	P_{29},C^{1}_{15} \cup P_{14},C^{1}_{15} \cup P_{2} \cup P_{4} \cup (\text{u}), C^{1}_{15} \cup P_{2} \cup P_{4} \cup (\text{h}), (\text{z}) \cup P_{14} \cup (\text{t}),\\(\text{z}) \cup P_{2} \cup P_{4} \cup (\text{u}) \cup (\text{t}), (\text{z}) \cup P_{2} \cup P_{4} \cup (\text{h}) \cup (\text{t})$\}.
	\item [Case 2.] $n\equiv3\pmod4$:
	
	Now, we investigate the mixed cospectral mates of $P_{4k+3}$.
	\begin{theorem}
		For every positive integer $k$, $P_{4k+3}$ is not DHS.
	\end{theorem}
	\begin{proof}
		One can see that the disjoint union of the Graphs $\text{G}_{2}$ and $P_{1}$ is a cospectral mate of $P_{7}$. Moreover, for every $k\geq 2$, the disjoint union of the Graphs $\text{G}_{k-1}^{2k}$ and $P_{k}$, as well as the disjoint union of the Graphs $ C^{1}_{2k+2} $ and $ P_{2k+1} $, is a cospectral mate of $P_{4k+3}$.
	\end{proof}
\end{enumerate}
In the next section, we investigate the mixed cospectral mates of mixed cycles.
\section{Spectral Characterization of Mixed Cycles}
In this section, we study the spectral determination problem of mixed cycles.

 Noting to the spectrum of the mixed cycle $ C_{2n} $, one can see that for $ n \geq 3 $, we have
\begin{equation*}
\text{Spec}_{H}(C_{2n})=\text{Spec}_{H}(C_{n} \cup C^{2}_{n}).
\end{equation*}
Also, for every integer $ n>2 $, the spectrum of the mixed cycles $ C^{2}_{2n} $ and $ C_{n}^{1} $ implies that $\text{Spec}_{H}(C^{2}_{2n})=\text{Spec}_{H}(C^{1}_{n} \cup C^{1}_{n})$. So, we have the following theorem.
\begin{theorem}\label{sadezoj}
	 Except $ C_{4} $ and $ C^{2}_{4} $, the mixed cycles $ C_{2n} $ and $ C_{2n}^{2} $ are not DHS, where $ n \geq 2 $.
\end{theorem}
\begin{remark}
	Here, we introduce two other mixed graphs which are cospectral mates of $ C_{2n} $ and $ C^{2}_{2n} $. By $ E_{r} $ and $ F_{r} $, we mean the mixed graph and signed graph depicted in Figure \ref{Qhk}, where $ G(E_{r})= G(F_{r})=\theta_{3,3,r} $. Using the switching function introduced in Remark \ref{remark8}, we have
	\begin{equation*}
	\text{Spec}_{H}(C^{2}_{2n})=\text{Spec}(C_{2n},-), \quad \text{Spec}_{H}(E_{r})=\text{Spec}(F_{r}). 
	\end{equation*}
	Now, by results in \cite{ss}, one can state the following theorem which is analogous to the signed graphs.
\end{remark}
\begin{figure}
	\centering
	\includegraphics[height=2.2cm]{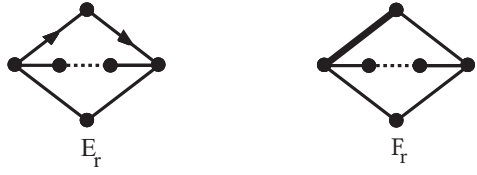}
	\caption{The mixed graph $E_{r}$ and the signed graph $ F_{r} $. The thick edge is negative.}
	\label{Qhk}
\end{figure}
\begin{theorem}
	For every integer $ r \geq 3 $, the graphs $ P_{r-1} \cup E_{r-1} $ and $ G_{r-2}^{r-2} $ are cospectral mates of $ C_{2r} $ and $ C^{2}_{2r} $, respectively.
\end{theorem}
In order to determine whether mixed cycles of Type $0$ of odd order are DHS or not, we need the following theorem and lemma.
\begin{theorem}\cite{liuli}\label{sym}
	If $ X $ is a mixed graph with no real mixed odd cycles, then its spectrum is symmetric about zero.
\end{theorem}
\begin{lemma}\cite[p.97]{god}\label{properinterlace}
	Let $ G $ be a connected graph. If $ H $ is a proper subgraph of $ G $, then $ \lambda_{1}(H) < \lambda_{1}(G) $.
\end{lemma}
\begin{theorem}
	For every positive integer $ n $, the mixed cycle $ C_{2n+1} $ is DHS.
\end{theorem}
\begin{proof}
	Recall that  \[\text{Spec}_{H}(C_{2n+1})=
	\begin{pmatrix}
	2 & 2\cos\dfrac{2\pi}{2n+1} & 2\cos\dfrac{4\pi}{2n+1} & \dots & 2\cos\dfrac{2n\pi}{2n+1} \\
	1 & 2 & 2 & \dots & 2
	\end{pmatrix}
	\]
	and if $ x \in \text{Spec}_{H}(C_{2n+1}) $, then $ -x \notin \text{Spec}_{H}(C_{2n+1}) $.
	
	Now, suppose that $ \text{Spec}_{H}(X)=\text{Spec}_{H}(C_{2n+1}) $ and let $ X_{1}, \ldots , X_{k} $ be the connected components of $ X $. Note that for $ i=1,\ldots , k $, $ X_{i} $ is not a tree, because the spectrum of every tree is symmetric about zero. Therefore, since $ \vert E(X) \vert =\vert V(X) \vert= 2n+1$, each $ X_{i} $ is a unicyclic graph. By Theorem \ref{sym}, each $ X_{i} $ is a mixed unicyclic graph of Type $0$ or $2$ whose cycle is of odd order. In addition, Theorem \ref{interlace} implies that none of $ X_{i} $ is a mixed unicyclic graph of Type 2, because otherwise, since $ -2 $ is an eigenvalue of odd cycles of Type 2, an eigenvalue which is at most $ -2 $ appears in the spectrum of $ X $, a contradiction. Finally, by Lemma \ref{properinterlace}, each $ X_{i} $ is a cycle of Type $ 0 $ and therefore $X$ is the disjoint union of one or more odd cycles of Type 0. If $X$ contains two or more cycles of Type $0$, then the multiplicity of $2$ is at least two, a contradiction. Hence, $X$ is $C_{2n+1}$, and the proof is complete.\smallskip
\end{proof}
Since $ \text{Spec}_{H}(C_{2n+1})=-\text{Spec}_{H}(C_{2n+1}^{2}) $, a similar argument works for odd cycles of Type 2. So, we have the following Theorem.
\begin{theorem}
	For every positive integer $ n $, the mixed cycle $ C_{2n+1}^{2} $ is DHS.
\end{theorem}
Now, we turn to this question: Is the cycle $ C^{1}_{n} $ DHS for each $ n $? Note that analogous to the paths, the eigenvalues of $ C_{n}^{1} $ are between $ -2 $ and $ 2 $ and they are simple. So, all the previous discussions for the paths hold. Therefore if $ X $ is a mixed cospectral mate of $ C^{1}_{n}$, then all possible components of $ X $ are  expressed in Figure \ref{acceptable}.

\begin{theorem}\label{type1even}
	Let $ n\geq2 $ be an integer. Then, the cycle $ C^{1}_{2n} $ is DHS if and only if $ n \neq 6 $.
\end{theorem}
\begin{proof}
	Let $ X $ be a mixed cospectral mate of $ C^{1}_{2n} $. Recall that 
	\begin{equation*}
	\text{Spec}_{H}(C^{1}_{2n})= \{2\cos \dfrac{2k+1}{4n}\pi , ~ k=0,1, \ldots , 2n-1 \}.
	\end{equation*}
	So, the eigenvalues of these cycles have the form $ 2\cos (\dfrac{p}{q}\pi) $, where $ \gcd(p,q)=1 $ and  $ 4 \vert q $. Also, note that $ 0 \notin \text{Spec}_H(C^{1}_{2n}) $. Since the Graphs (w), $ (\text{D}_{l}) $ and paths of odd order have a zero eigenvalue, they cannot be components of $ X $. Also, the spectrum of a path of even order, say $ 2m $, contains $ 2 \cos \dfrac{\pi}{2m+1} $ which is not an eigenvalue of $ C^{1}_{2n} $. Finally, the Graphs (y) and (z) have the eigenvalue $ 2\cos\dfrac{\pi}{3} $ and $ 2\cos\dfrac{\pi}{30} $, respectively, which are not in the spectrum of $ C^{1}_{2n} $. This implies that none of the components of $ X $ is a tree and therefore, all components of $ X $ are unicyclic.
	
	Similarly, by considering the spectrum of unicyclic graphs given in Figure \ref{acceptable}, one can see that only the Graphs (q), (s), $ (\text{G}_{t}) $, $ (\text{G}_{t}^{t+m}) $, (k) and even cycles of Type 1 can be components of $ X $. Now, by Equation (\ref{eccpm}), the constant term of the characteristic polynomial of  $ C^{1}_{2n} $ is either $-2$ or $2$, while the constant term of the characteristic polynomial of the Graphs $ (\text{G}_{t}) $ and $ (\text{G}_{t}^{t+m}) $ belongs to $\{-4,0,4\}$. Therefore, $ (\text{G}_{t}) $ and $ (\text{G}_{t}^{t+m}) $ cannot be components of $ X $. Also, if for some $ m $ ($ m < n $), $ C^{1}_{2m} $ is a component of $ X $, then for some $k \geq 1$ we have,
	\begin{equation*}
	2\cos \dfrac{\pi}{4m}=2\cos \dfrac{(2k+1)\pi}{4n}.
	\end{equation*} 
	So, $ n=(2k+1)m $, hence $ m\leq \frac{n}{3} $. Note that at most one even cycle of Type 1 can be a component of $ X $, because otherwise the constant term of the characteristic polynomial of $ X $ would be a multiple of 4, a contradiction.
	
	Since the Graphs (q), (s) and (k) are of order 8,$\,$8 and 4, respectively, we should have
	\begin{equation*}
	8+8+4+\frac{2n}{3} \geq 2n.
	\end{equation*}  
	Therefore, $ n \leq 15 $. In other words, if $ n\geq 16 $, then $ C^{1}_{2n} $ is DHS.
	
	Obviously, for $ m < n $, $\lambda_{1}(C^{1}_{2m}) <  \lambda_{1}(C^{1}_{2n}) $. Therefore, $ \lambda_{1}(C^{1}_{2n})$ should be an eigenvalue of one of the Graphs (q), (s) and (k). Now, the spectrum of these graphs shows that only $ C^{1}_{12} $ has a cospectral mate, which is the Graph $  C^{1}_{4} \cup \text{(q)}  $. This completes the proof.
\end{proof}
We close this paper by characterizing the odd cycles of Type $1$ which are DHS.
\begin{theorem}
	Let $ n $ be a positive integer. Then, the cycle $ C^{1}_{2n+1} $ is DHS if and only if $ n \notin \lbrace 3,4,7 \rbrace $.
\end{theorem}
\begin{proof}
	Let $ X $ be a mixed cospectral mate of $ C^{1}_{2n+1} $. Recall that 
	\begin{equation*}
	\text{Spec}_{H}(C^{1}_{2n+1})= \{2\cos \dfrac{2k+1}{4n+2}\pi , ~ k=0,1, \ldots , 2n \}.
	\end{equation*}
	Similar to the discussions in the proof of Theorem \ref{type1even}, by considering the spectrum of the graphs in Figure \ref{acceptable}, we find that the only possible components of $ X $ are the Graphs (p), (t), (v), $ (\text{G}_{t}) $, $ (\text{G}_{t}^{t+m}) $, (w), (z), $ (\text{D}_{l}) $, paths of odd order and cycles of Type 1. Note that $ \lambda_{1}(C^{1}_{2n+1})= 2\cos \dfrac{\pi}{4n+2} $ and two Graphs $ (\text{G}_{t}) $ and $ (\text{G}_{t}^{t+m}) $ have $ t+4 $ and $ 2t+m+4 $ vertices, respectively, which are at most $ \vert V(C_{2n+1}^{1}) \vert = 2n+1 $. Notice that $ \lambda_{1}(\text{G}_{t})= 2\cos \dfrac{\pi}{2t+4} < 2\cos \dfrac{\pi}{2(t+4)}\leq 2\cos \dfrac{\pi}{4n+2} $ and $ \lambda_{1}(\text{G}_{t}^{t+m})=2\cos \dfrac{\pi}{2t+2m+4} < 2\cos \dfrac{\pi}{2(2t+m+4)}\leq 2\cos \dfrac{\pi}{4n+2}$. So, $ \lambda_{1}(C^{1}_{2n+1})$ cannot be an eigenvalue of these two family of graphs. The same discussion holds for $ (\text{D}_{l}) $. Hence, $ \lambda_{1}(C^{1}_{2n+1})$ appears in the spectrum of one of the Graphs (p), (t), (v), (w) and (z). Now, considering the spectrum of these graphs, it is seen that if $ 2n+1 \geq 17 $, then $ C^{1}_{2n+1} $ is DHS.
	
	Now, the spectrum of the mentioned graphs implies that $ C^{1}_{3} $, $ C^{1}_{5} $, $ C^{1}_{11} $ and $ C^{1}_{13} $ are DHS. Also, the class of all mixed cospectral mates of $ C^{1}_{7} $, $ C^{1}_{9} $  and $ C^{1}_{15} $ is as follows.
	
	$\qquad \qquad \qquad \qquad \quad \{C^{1}_{7},\text{(p)}\}, \quad \{C^{1}_{9}, \text{(v)} \cup P_{1}\}, \quad \{C^{1}_{15},\text{(z)} \cup \text{(t)}\}.$
\end{proof}
\newpage
\appendix
\section*{Appendix  A}
	\begin{figure}[!h]
		\centering
		$\includegraphics[height=13cm]{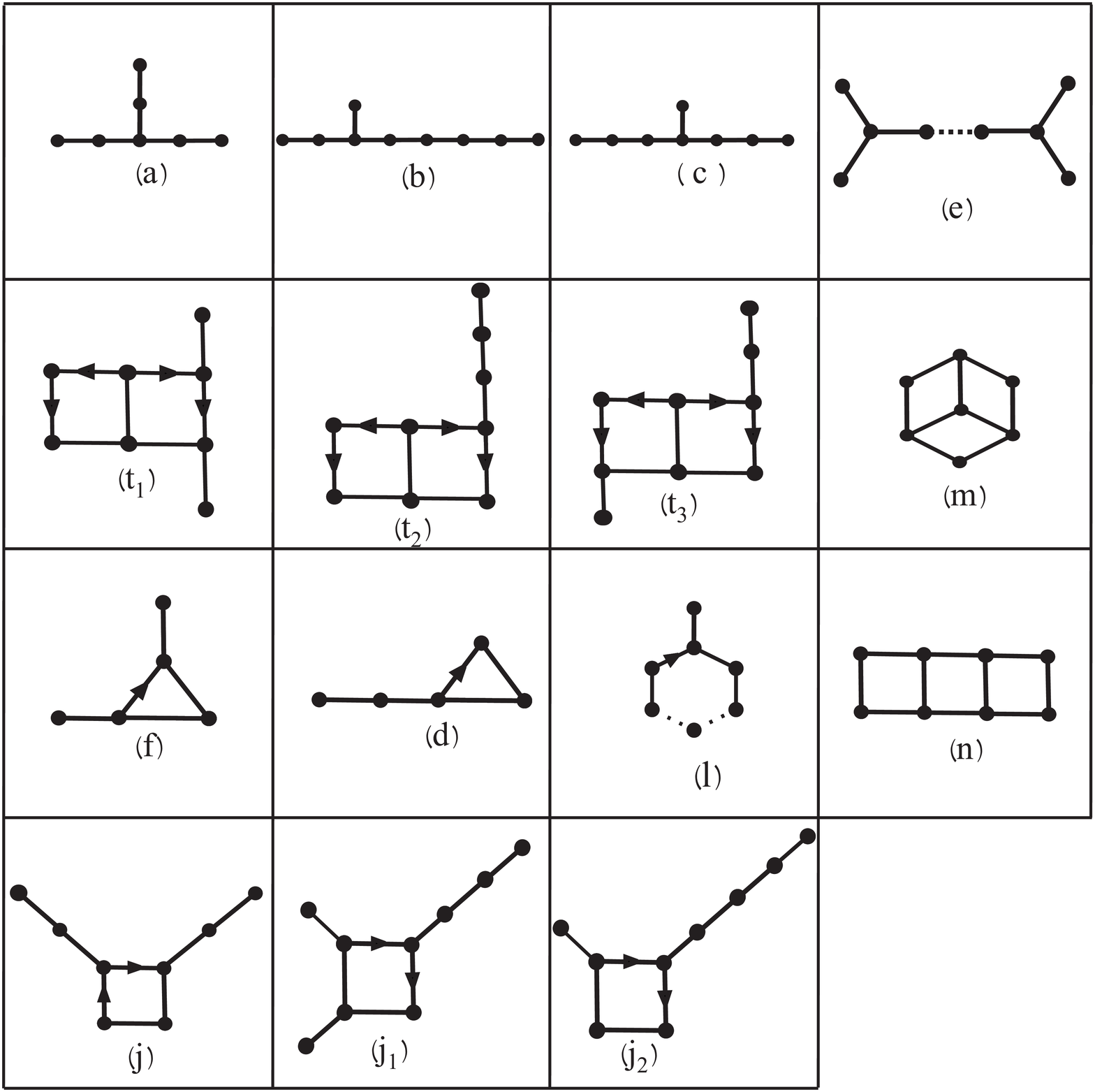}$
		\caption{Mixed $(-2,2)$-out graphs: Using a computer search, we find that all mixed graphs with the underlying Graph (m) and (n) are $(-2,2)$-out. It is shown in \cite{smith} that the Graphs (a), (b), (c) and (e) are $(-2,2)$-out. The Graph (l) has order $ t $, where $ 5 \leq t \leq 8 $.}
		\label{unacceptable}
	\end{figure}
\newpage
	\begin{figure}[!h]
		\centering
		\includegraphics[height=17cm]{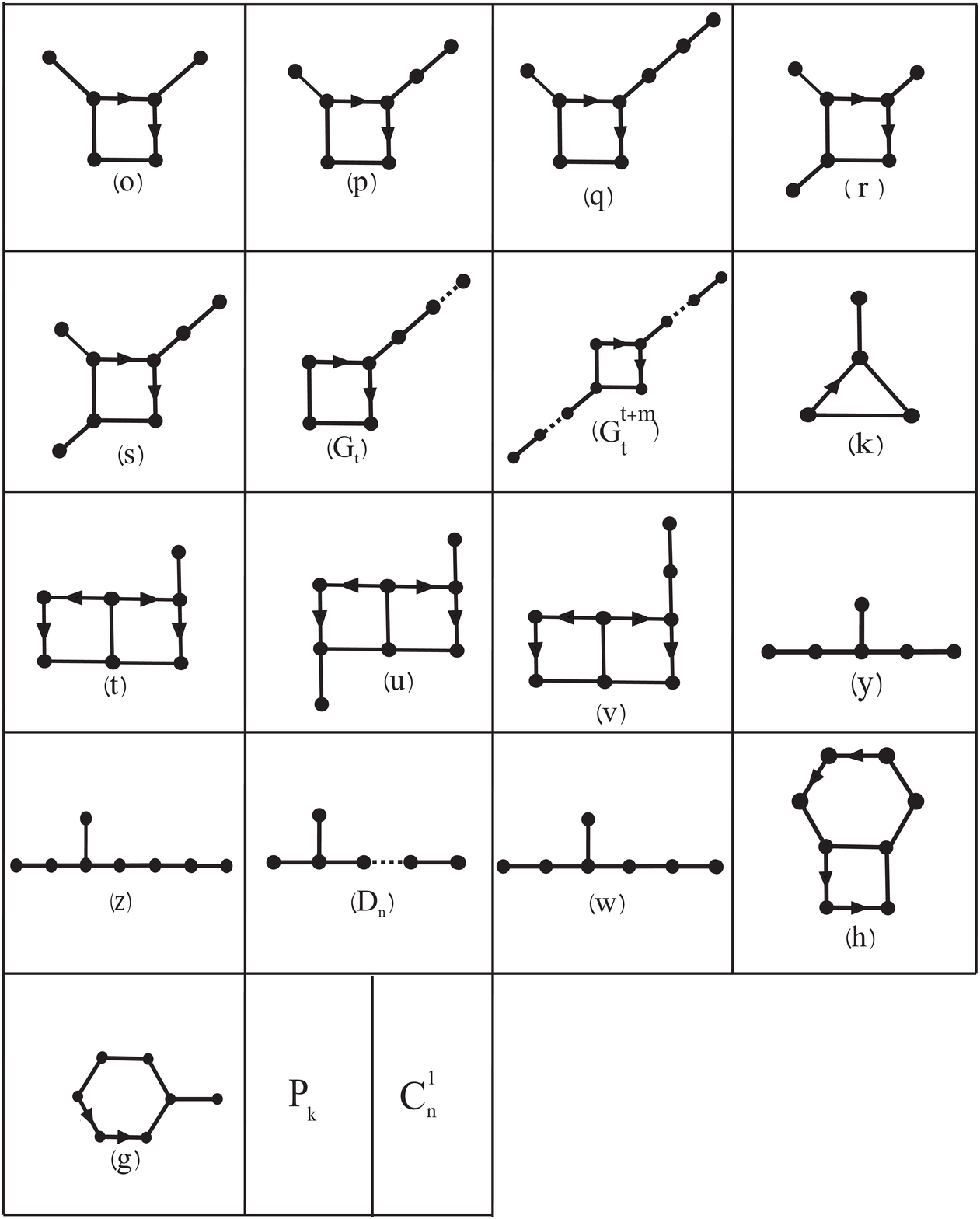}
		\caption{Admissible graphs: All  eigenvalues  of these graphs are in  $(-2,2)$ and  simple.}
		\label{acceptable}
	\end{figure}
	The spectrum of the graphs in Figure \ref{acceptable} are as follows.\\
		\vspace*{0.4cm}
	$\textrm{Spec}_{H}(\text{o})=\{2cos\dfrac{\pi}{9},2cos\dfrac{2\pi}{9},2cos\dfrac{4\pi}{9},2cos\dfrac{5\pi}{9},2cos\dfrac{7\pi}{9},2
	cos\dfrac{8\pi}{9}\}$\\
	\vspace*{0.4cm}
	$\textrm{Spec}_{H}(\text{p})=\{2cos\dfrac{\pi}{14}, 2cos\dfrac{3\pi}{14}, 2cos\dfrac{5\pi}{14},2cos\dfrac{7\pi}{14}, 2cos\dfrac{9\pi}{14}, 2cos\dfrac{11\pi}{14}, 2cos\dfrac{13\pi}{14}\}$ \\
	\vspace*{0.4cm}
	$\textrm{Spec}_{H}(\text{q})=\{2cos\dfrac{\pi}{24},2cos\dfrac{5\pi}{24},2cos\dfrac{7\pi}{24},2cos\dfrac{11\pi}{24},2cos\dfrac{13\pi}{24},2cos\dfrac{17\pi}{24},2cos\dfrac{19\pi}{24},2cos\dfrac{23\pi}{24}\}$\\
	\vspace*{0.4cm}
	$\textrm{Spec}_{H}(\text{r})=\{2cos\dfrac{\pi}{12},2cos\dfrac{2\pi}{12},2cos\dfrac{5\pi}{12},2cos\dfrac{6\pi}{12},2cos\dfrac{7\pi}{12},2cos\dfrac{10\pi}{12},2cos\dfrac{11\pi}{12}\}$\\
	\vspace*{0.4cm}
	$\textrm{Spec}_{H}(\text{g})=\{2cos\dfrac{\pi}{12},2cos\dfrac{2\pi}{12},2cos\dfrac{5\pi}{12},2cos\dfrac{6\pi}{12},2cos\dfrac{7\pi}{12},2cos\dfrac{10\pi}{12},2cos\dfrac{11\pi}{12}\}$\\
	\vspace*{0.4cm}
	$\textrm{Spec}_{H}(\text{s})=\{2cos\dfrac{\pi}{20},2cos\dfrac{3\pi}{20},2cos\dfrac{7\pi}{20},2cos\dfrac{9\pi}{20},2cos\dfrac{11\pi}{20},2cos\dfrac{13\pi}{20},2cos\dfrac{17\pi}{20},2cos\dfrac{19\pi}{20}\}$\\
	\vspace*{0.4cm}
	$\textrm{Spec}_{H}(\text{t})=\{2cos\dfrac{\pi}{10},2cos\dfrac{\pi}{6},2cos\dfrac{3\pi}{10},2cos\dfrac{5\pi}{10},2cos\dfrac{7\pi}{10},2cos\dfrac{5\pi}{6},2cos\dfrac{9\pi}{10}\}$\\
	\vspace*{0.4cm}
	$\textrm{Spec}_{H}(\text{u})=\{2cos\dfrac{\pi}{15},2cos\dfrac{2\pi}{15},2cos\dfrac{4\pi}{15},2cos\dfrac{7\pi}{15},2cos\dfrac{8\pi}{15},2cos\dfrac{11\pi}{15},2cos\dfrac{13\pi}{15},2cos\dfrac{14\pi}{15}\}$\\
	\vspace*{0.4cm}
	$\textrm{Spec}_{H}(\text{h})=\{2cos\dfrac{\pi}{15},2cos\dfrac{2\pi}{15},2cos\dfrac{4\pi}{15},2cos\dfrac{7\pi}{15},2cos\dfrac{8\pi}{15},2cos\dfrac{11\pi}{15},2cos\dfrac{13\pi}{15},2cos\dfrac{14\pi}{15}\}$\\
	\vspace*{0.4cm}
	$\textrm{Spec}_{H}(\text{k})=\{2cos\dfrac{\pi}{12},2cos\dfrac{5\pi}{12},2cos\dfrac{7\pi}{12},2cos\dfrac{11\pi}{12}\}$\\
	\vspace*{0.4cm}
	$\textrm{Spec}_{H}(\text{v})=\{2cos\dfrac{\pi}{18},2cos\dfrac{3\pi}{18},2cos\dfrac{5\pi}{18},2cos\dfrac{7\pi}{18},2cos\dfrac{11\pi}{18},2cos\dfrac{13\pi}{18},2cos\dfrac{15\pi}{18},2cos\dfrac{17\pi}{18}\}$\\
	\vspace*{0.4cm}
	$\textrm{Spec}_{H}(\text{G}_{t})=\{2cos\dfrac{(2k+1)\pi}{2t+4},k=0,\ldots,t+1\}\cup \{2\cos\dfrac{\pi}{4},2\cos\dfrac{3\pi}{4}\}$\\
	\vspace*{0.4cm}
	$\textrm{Spec}_{H}(\text{G}^{t+m}_{t})=\{2cos\dfrac{(2k+1)\pi}{2t+2m+4},k=0,\ldots,t+m+1\}\cup \{2cos\dfrac{(2k+1)\pi}{2t+4},k=0,\ldots,t+1\}$\\
	\vspace*{0.4cm}
	$\textrm{Spec}_{H}(\text{y})=\{2cos\dfrac{\pi}{12},2cos\dfrac{4\pi}{12},2cos\dfrac{5\pi}{12},2cos\dfrac{7\pi}{12},2cos\dfrac{8\pi}{12},2cos\dfrac{11\pi}{12}\}$\\
	\vspace*{0.4cm}
	$\textrm{Spec}_{H}(\text{z})=\{2cos\dfrac{\pi}{30},2cos\dfrac{7\pi}{30},2cos\dfrac{11\pi}{30},2cos\dfrac{13\pi}{30},2cos\dfrac{17\pi}{30},2cos\dfrac{19\pi}{30},2cos\dfrac{23\pi}{30},2cos\dfrac{29\pi}{30}\}$
	\\
	\vspace*{0.4cm}
	$\textrm{Spec}_{H}(\text{D}_n)=\{0\}\cup \{2cos\dfrac{(2k+1)\pi}{2n-2}:k=0,1,\ldots,n-2\}$ \cite[p.38]{haemers}
	\\
	\vspace*{0.4cm}
	$\textrm{Spec}_{H}(\text{w})=\{2cos\dfrac{\pi}{18},2cos\dfrac{5\pi}{18},2cos\dfrac{7\pi}{18},2cos\dfrac{9\pi}{18},2cos\dfrac{11\pi}{18},2cos\dfrac{13\pi}{18},2cos\dfrac{17\pi}{18}\}$
	\\
	\vspace*{0.4cm}
	$\textrm{Spec}_{H}(P_k)=\{2cos\dfrac{r\pi}{k+1},r=1,\ldots,n\}$, \quad $\textrm{Spec}_{H}(C_{n}^{1})=\{2cos\dfrac{(2k+1)\pi}{2n}; k=0,\ldots,n-1\}$.	\\
	
\end{document}